\documentclass[11pt,a4,fleqn,english]{amsart}

\usepackage{amsmath}   
\usepackage{amsfonts,amssymb}
\usepackage[height=22cm , width = 16cm , top = 4cm , left = 3cm, a4paper]{geometry}
\usepackage[a4paper]{geometry}

\usepackage{graphicx} 
\usepackage{url}      

\usepackage{babel}
\usepackage{latexsym}

\input colordvi

\theoremstyle{plain}

\newtheorem{theo}{Theorem}[section]
\newtheorem{definition}[theo]{Definition}

\theoremstyle{remark}
\newtheorem{lemma}[theo]{Lemma}
\newtheorem{hypp}[theo]{Hypotheses}

\numberwithin{equation}{section}

\title[Uniqueness for coagulation-multiple fragmentation]{On the uniqueness for coagulation and multiple fragmentation equation}
\author{Ankik Kumar Giri}

\thanks{Johann Radon Institute for Computational and Applied Mathematics (RICAM), Austrian Academy of Sciences, 
Altenbergerstrasse\ 69, A-4040 Linz, Austria} \email{ankik.giri@ricam.oeaw.ac.at, ankik.math@gmail.com}

\begin{document}
 \baselineskip 16pt \markright{Text}

\thispagestyle{empty}


\begin{abstract}
In this article, the uniqueness of weak solutions to the continuous coagulation and multiple fragmentation equation is proved for a 
large range of unbounded coagulation and multiple fragmentation kernels. The multiple fragmentation kernels may have a singularity at origin. This work generalizes the preceding ones, by including some physically relevant coagulation and fragmentation kernels which were not considered before.
\end{abstract}

\maketitle

 \thispagestyle{empty}

\section{Introduction}\label{intro}

We analyze the following continuous coagulation and multiple fragmentation equation 
\begin{align}\label{cfe1}
\frac{\partial f(x,t)}{\partial t}  = & \frac{1}{2}\int_{0}^{x} K(x-y,y)f(x-y,t)f(y,t)dy - 
\int_{0}^{\infty} K(x,y)f(x,t)f(y,t)dy\nonumber\\  
&+ \int_{x}^{\infty} b(y,x)S(y)f(y,t)dy-S(x)f(x,t),
\end{align}
with
\begin{eqnarray}\label{in1}
f(x,0) = f_{0}(x)\geq 0,
\end{eqnarray}

where $f(x,t)$ is the number density of particles of size $x \in \mathbb{R}_{>0}:=(0,\infty)$ at time $t\geq0$. 
The first two integrals on the right-hand side of (\ref{cfe1}) describe, 
respectively, the birth and death of particles of size $x$ due to the coagulation 
process. The coagulation kernel $K(x,y)$ represents the rate at which particles of size 
$x$ coalesce with particles of size $y$. The remaining two integrals on the right-hand side of (\ref{cfe1}), 
due to the fragmentation process, can be interpreted similarly. Here the particles can fragment into more than two 
pieces. The breakage function $b(x,y)$ is the 
probability density function for the formation of particles of size $y$ from 
the breakage of particles of size $x$. Moreover, it is assumed that $b$ is non-negative measurable function which satisfies 
$b(x,y)=0$ for  $x< y$ and $S$ is the fragmentation rate. The fragmentation rate $S$ and 
breakage function $b$ can be expressed in terms of the multiple-fragmentation 
kernel $\Gamma$ as follows
\begin{eqnarray}\label{Selection rate defi1}
S(x)=\int_{0}^{x}\frac{y}{x}\ \Gamma(x,y)dy,\ \ \ \ b(x,y)=\Gamma(x,y)/ S(x).
\end{eqnarray}

The number of fragments obtained from the breakage of particles of size $y$,
\begin{eqnarray}\label{N1}
\int_{0}^{y}b(y,x)dx = N< \infty, \ \ \mbox{for all}\ \ y\in \mathbb{R}_{>0},
\end{eqnarray}
and the necessary condition for mass conservation is
\begin{eqnarray}\label{mass1}
\int_{0}^{y}xb(y,x)dx = y\ \ \text{for all}\ \ y\in \mathbb{R}_{>0}.
\end{eqnarray}

The typical question of existence and uniqueness of solutions to coagulation-fragmentation equation has already been discussed 
extensively in several articles for example \cite{Aizenman:1979, Ball:1990, Costa:1995, Dubovski:1996, Escobedo:2003, Giri:2011II, Laurencot:2004, Mischler:2002, Stewart:1990I, Stewart:1990II}. However, most of these articles are dedicated to the case of 
binary fragmentation. The case of multiple fragmentation was not considered up to that level. Best to our knowledge, the first 
existence and uniqueness of solutions to the continuous coagulation and multiple fragmentation equation was proved 
in \cite{Melzak:1957} for bounded coagulation and fragmentation kernels $K$ and $F$ respectively. By using a different 
(semigroup) approach, a similar result was acquired in \cite{McLaughlin:1997II}. In the similar contest, a more relaxed result 
on existence and uniqueness of solutions to (\ref{cfe1})-(\ref{in1}) is shown in \cite{Lamb:2004} when $S$ satisfies almost a linear growth but 
still with a bounded $K$. The case of unbounded coagulation and fragmentation kernels is later considered in 
\cite{Laurencot:2000, Giri:2011I, Giri:2012, Banasiak:2012} where the existence of solutions to (\ref{cfe1})-(\ref{in1}) is demonstrated 
under different growth conditions on coagulation and fragmentation kernels. In \cite{Laurencot:2000}, the coagulation kernel $K$ 
of the type $K(x,y)=r(x)r(y)$ with no growth restriction on $r$ is assumed with a reasonable growth condition on multiple fragmentation 
kernel $\Gamma$. In \cite{Giri:2011I}, the coagulation kernel $K$ satisfying $K(x,y)\leq \phi(x)\phi(y)$ for some sub-linear 
function $\phi$ and a reasonable growth condition on $\Gamma$ are assumed to show the existence of solutions. Still the fragmentation 
kernel $\Gamma$ is required to be bounded near origin in \cite{Laurencot:2000, Giri:2011I} which thus excludes classical fragmentation 
kernels discussed in the literature such as $\Gamma(y,x)=(\alpha+2)x^{\alpha}y^{\gamma-(\alpha+1)}$ with $\alpha>-2$ and 
$\gamma\in \mathbb{R}$, see \cite{McGradyZiff:1987}. Later, this type of fragmentation kernels are partially included 
in \cite{Giri:2012} which 
extends the previous results with the same growth conditions on the coagulation kernel $K$ as in \cite{Giri:2011I}. These types of 
kernels are also recently discussed in \cite{Banasiak:2012} but with an assumption of finiteness on higher moments.
However, the uniqueness of solutions to (\ref{cfe1})-(\ref{in1}) for unbounded coagulation and fragmentation kernels
is only shown in \cite{Giri:2011I} which does not include most of the physically relevant coagulation and fragmentation kernels considered in the existence result \cite{Giri:2011I, Giri:2012}. 

The purpose of this work is to prove the uniqueness of solutions to (\ref{cfe1})-(\ref{in1}) which can cover this gap partially. Let us 
briefly outline the manuscript. In this section, we give some hypotheses, notation of spaces, definitions  and state the main result in 
Theorem \ref{unique theorem} to demonstrate the uniqueness of solutions. A few examples of unbounded coagulation and multiple fragmentation
kernels are also given from \cite{Giri:2012} at the end of the section which are covered in the present work. Finally, Theorem \ref{unique theorem} is proved by 
showing the integrability of higher moments in Theorem \ref{integrability thm} in Section \ref{uni}. The integrability of some higher moments had been previously proved in \cite{Carr:1992} for sublinear coagulation kernels in the discrete coagulation-fragmentation equations. An inspiration to complete this work came from 
\cite{Stewart:1990II, Costa:1995, Giri:2011II}.

In particular, we make the following hypotheses on the kernels.
\begin{hypp}
 (A1) $K$ is a non-negative measurable function on $\mathbb{R}_{\geq 0} \times \mathbb{R}_{\geq 0}$ 
and is symmetric, i.e. $K(x,y)=K(y,x)$ for all $x,y \in \mathbb{R}_{> 0}$,

(A2) $K(x,y)\leq \phi(x)\phi(y)$ for all $x,y\in \mathbb{R}_{> 0}$ 
where $\phi(x)\leq k_1(1+x)^{\mu}$ for some $0\leq \mu < 1$ and constant $k_1>0$.

(A3) $\Gamma$  is a non-negative measurable function on $\mathbb{R}_{> 0} \times \mathbb{R}_{> 0}$ such that $\Gamma(x,y)=0$ if $0<x<y$. 
Defining $S$ and $b$ by (\ref{Selection rate defi1}), we assume that $b$ satisfies  (\ref{N1})-(\ref{mass1}).

(A4) $S:\mathbb{R}_{> 0}\to \mathbb{R}_{\geq 0}$ is a measurable function. For all $x>0$, there exist a constant $m>0$ such that $S(x)\leq m_1(1+x)^{1-\lambda}$, where $\lambda \in ]0,1[$.

(A5) There are constants $L>0$ and $1+\nu>0$ such that
\begin{eqnarray*}
 \Gamma(x,y)=b(x,y)\cdot S(x) \geq R(x):=L(1+x)^{\nu} \hspace{.2cm} \mbox{for any} \hspace{.2cm} x\geq 1 \hspace{.2cm}  \mbox{and} \hspace{.2cm} y\in (0,x).
\end{eqnarray*}
\end{hypp}

Let $X$ be the following Banach space with norm $\|\cdot\|$
\begin{eqnarray*}
X=\{f\in L^1]0,\infty[:\|f\|< \infty\}\ \ \mbox{where}\ \ \|f\|=\int_{0}^{\infty}(1+x)|f(x)|dx
\end{eqnarray*}
and set
\begin{eqnarray*}
X^+=\{f\in X: f\geq 0 \ \ a.e.\}.
\end{eqnarray*}

The main result of this work is the following uniqueness result:
\begin{theo}\label{unique theorem}
Let $f$ be a solution of equation (\ref{cfe1})-(\ref{in1}) with initial data $f_0 \in X^+$. 
If (A1)-(A5) and $1+\nu>\mu$ hold, then the solution is unique.
\end{theo}

The $r$th moment of the number density distribution $f$ if it exists is defined by
\begin{eqnarray*}
M_{r}(t)= M_{r}(f(t)):= \int_{0}^{\infty} x^{r}f(x,t)dx, \ \ r\geq 0.
\end{eqnarray*}
The first two moments represent some important properties of the distribution. 
The zeroth ($r=0$) and first ($r=1$) moments are proportional to the total 
number and the total mass of particles, respectively.

\begin{definition}\label{def1} Let $T \in ]0,\infty]$. A solution $f$ of (\ref{cfe1})-(\ref{in1}) is a non-negative function $f: [0,T[\to X^+$ such that, for a.e. $x\in ]0,\infty[$ and all $t\in [0,T[$,

      (i) $s\mapsto f(x,s)$ is continuous on $[0,T[$,

     (ii)  the following integrals are finite
     \begin{align*}
     \int_{0}^{t}\int_{0}^{\infty}K(x,y)f(y,s)dyds<\infty\ \ and \ \ \int_{0}^{t}\int_{x}^{\infty}b(y,x)S(y)f(y,s)dyds<\infty,
     \end{align*}

      (iii)  the function $f$ satisfies the following weak formulation of (\ref{cfe1}){-(\ref{in1})}
     \begin{align*}
     f(x,t)&=f_0(x)+\int_{0}^{t}\left\{ \frac{1}{2}\int_{0}^{x}K(x-y,y)f(x-y,s)f(y,s)dy \right.\nonumber\\
     &\left. -\int_{0}^{\infty}K(x,y)f(x,s)f(y,s)dy +\int_{x}^{\infty}b(y,x)S(y)f(y,s)dy - S(x)f(x,s)\right\}ds.
     \end{align*}
\end{definition}

Let us mention the following coagulation kernels discussed in \cite{Banasiak:2012, Giri:2011I, Giri:2012} to show the existence 
of solutions to (\ref{cfe1})-(\ref{in1}). We should point out these kernels will satisfy the hypotheses (A1) and (A2). 

$(1)$ Shear kernel (non-linear velocity profile), see \cite{Aldous:1999, Smit:1994} 
\begin{eqnarray*}
 K(x,y)=k_0(x^{1/3}+y^{1/3})^{7/3}.
\end{eqnarray*}
$(2)$ The modified Smoluchowski kernel, see \cite{Koch:2007}, 
\begin{eqnarray*}
 K(x,y)=k_0\frac{(x^{1/3}+y^{1/3})^2}{x^{1/3}\cdot y^{1/3}+c}, \ \ c>0.
\end{eqnarray*}
The third example of coagulation kernel mentioned in \cite{Giri:2011I} is bounded at infinity and 
therefore is already covered in the previous work. It is also clear that the coagulation kernels satisfying 
$K(x,y)\leq x^{\mu_1}y^{\mu_2}+x^{\mu_2}y^{\mu_1}$ for some $\mu_1, \mu_2 \in [0,1[$ which are usually used in the mathematical 
literature satisfy (A1)-(A2).

Now let us take the following example of multiple fragmentation kernels partially considered in \cite{Giri:2012}. 
 \begin{eqnarray}
S(y) =y^\gamma \ \ \text{and}\ \ b(y,x) = \frac{\alpha+2}{y} \left( \frac{x}{y} \right)^\alpha\ \ \text{for}\ \ 0<x<y, \label{volvic}
\end{eqnarray}
 where  $\gamma \in \mathbb{R}$ and $\alpha > -2$, see \cite{McGradyZiff:1987,Peterson:1986}. Since this breakage function has a 
physical meaning only if $-2<\alpha\leq 0$. At $\alpha=0$, this gives us the case of binary fragmentation.
 Assume that $0<\gamma< 1$ and $-1< \alpha\leq 0$. (\ref{N1})-(\ref{mass1}) are clearly satisfied with 
$N=(\alpha+2)/(\alpha+1)$ and
\begin{eqnarray*}
 \Gamma(y,x)= (\alpha+2) x^{\alpha} y^{\gamma-(\alpha+1)} \ \ \mbox{for}\ \ 0<x<y.
\end{eqnarray*} 
First, to check (A5), we interchange the roles of $x$ and $y$ to obtain 
\begin{align*}
\Gamma(x,y)&= (\alpha+2) y^{\alpha} x^{\gamma-(\alpha+1)} \ \ \mbox{for}\ \ 0<y<x\\
&=(\alpha+2)x^{\gamma-1}(\frac{y}{x})^{\alpha}\\
 &\geq (\alpha+2)(1+x)^{\gamma-1}\\
 &=:L(1+x)^{\nu} \ \ \mbox{where}\ \ \nu=\gamma-1>-1\ \ \mbox{and}\ \ L=\alpha+2.
 \end{align*} 
 This shows that (A5) is satisfied. The hypothesis (A4) is obvious. 

\section{Uniqueness}\label{uni}

To prove the Theorem \ref{unique theorem}, we need to prove the following theorem motivated by \cite{Carr:1992, Giri:2011II}.

%
%

\begin{theo}\label{integrability thm} 
Let us assume that (A1)-(A5), and $1+\nu > \mu$ hold. 
Let $f\in X^+$ be any solution of equation (\ref{cfe1})-(\ref{in1}) on $[0,T[$, 
$T>0$. Then, for every $t\in [0,T[$ and for every $\delta>0$, $I_{2+\nu-\delta}(t)< \infty$, where $I_r(t):=\int_{0}^{t}M_r(f(s))ds$.
\end{theo}

Theorem \ref{integrability thm} can be proved by applying a repeated application of the following Lemma:

\begin{lemma} \label{lemma for integrability}
 Assume (A1)-(A5) and $1+\nu > \mu$ hold. Let $f\in X^+$ be any solution of equation (\ref{cfe1})-(\ref{in1}) 
on $[0,T[$, $T>0$, and assume $I_{\rho}(t)< \infty$ for all $t \in[0,T[$ and some $\rho \geq 1$ with $\rho>\mu$. 
Then, for all $t \in[0,T[$, $I_{\rho + \nu - \mu+1}(t)< \infty$ if $\rho-\mu<1$. In case $\rho-\mu \geq 1$ we obtain 
$I_{2+\nu-\delta}(t)<\infty$ for any $\delta>0$.
\end{lemma}
 \begin{proof}
Let us begin with the Definition \ref{def1} (iii), multiply by $x^\lambda$, $\lambda \in ]0,1[$ on the both sides 
and then integrate with respect to $x$ from $0$ to $n$ to get 
\begin{align*}
 \int_{0}^{n}x^{\lambda}[&f(x,t)-f_0(x)]dx \\
=& \int_{0}^{t}\bigg[\frac{1}{2} \int_{0}^{n}\int_{0}^{x} x^{\lambda} K(x-y,y)f(x-y,s)f(y,s)dydx\\
&-\int_{0}^{n}\int_{0}^{\infty} x^{\lambda} K(x,y)f(x,s)f(y,s)dydx\\
&+ \int_{0}^{n}\int_{x}^{\infty} x^{\lambda} b(y,x)S(y)f(y,s)dydx -\int_{0}^{n}x^{\lambda}S(x)f(x,s)dx\bigg]ds.
\end{align*}
Changing the order of integrations in the first and third terms on the right-hand side, substituting $x-y=x'$, $y=y'$ in the first one, 
and then interchanging the roles of $x$ and $y$ in both the terms, we obtain 
\begin{align}\label{1st eq lambda}
 \int_{0}^{n}x^{\lambda}[&f(x,t)-f_0(x)]dx \nonumber\\
& +\int_{0}^{t}\bigg[\frac{1}{2} \int_{0}^{n}\int_{0}^{n-x} \{x^{\lambda}+y^{\lambda}-(x+y)^{\lambda}\} K(x,y)f(x,s)f(y,s)dydx\nonumber\\
&+\int_{0}^{n}\int_{n-x}^{\infty} x^{\lambda} K(x,y)f(x,s)f(y,s)dydx\bigg]\nonumber\\
=& \int_{0}^{t}\bigg[\int_{0}^{n}\int_{0}^{x} y^{\lambda} b(x,y)S(x)f(x,s)dydx \nonumber\\
&+ \int_{n}^{\infty}\int_{0}^{n}  y^{\lambda} b(x,y)S(x)f(x,s)dydx-\int_{0}^{n}x^{\lambda}S(x)f(x,s)dx\bigg]ds.
\end{align}
Since $0< \lambda <1$ and $f \in X^+$, the first term on the left-hand side in (\ref{1st eq lambda}) is bounded 
independently of $n$ and is convergent as $n \to \infty$.

Let us now estimate the last term on the left-hand side in (\ref{1st eq lambda}) as 
\begin{align*}
 \int_{0}^{t}\int_{0}^{n}\int_{n-x}^{\infty} & x^{\lambda}K(x,y)f(x,s)f(y,s)dydxds\\ 
\leq&\int_{0}^{t}\int_{0}^{\infty}\int_{0}^{\infty}  x^{\lambda}K(x,y)f(x,s)f(y,s)dydxds\\
\leq & 2^{\mu}k_1^2\int_{0}^{t}\int_{0}^{\infty}\int_{0}^{1}  x^{\lambda}(1+y)f(x,s)f(y,s)dxdyds\\
&+ 2^{\mu} k_1^2\int_{0}^{t}\int_{0}^{\infty}\int_{1}^{\infty}  x^{\lambda+\mu}(1+y)f(x,s)f(y,s)dxdyds.
\end{align*}
In case $\lambda + \mu \leq \rho$ we obtain
\begin{align}
  \int_{0}^{t}\int_{0}^{n}\int_{n-x}^{\infty} & x^{\lambda}K(x,y)f(x,s)f(y,s)dydxds\nonumber\\
& \leq 2^{\mu} k_1^2 \max _{s \in [0,t]} \|f(s)\|\int_{0}^{t}[M_{\lambda}(f(s)) + M_{\rho}(f(s))]ds < \infty.
\end{align}
This shows that the last term on the right-hind side in (\ref{1st eq lambda}) is bounded independently of $n$. 
Therefore, it is convergent as $n \to \infty$.

Next, we consider the second term on the left-hand side of equation (\ref{1st eq lambda}).  Using an important inequality $(x+y)^{\lambda} \leq x^{\lambda}+y^{\lambda}$ for some $\lambda\in [0,1]$, we obtain 
\begin{align*}
\frac{1}{2}\int_{0}^{t}\int_{0}^{n}\int_{0}^{n-x}&\{x^{\lambda}+y^{\lambda}-(x+y)^{\lambda}\}K(x,y)f(x,s)f(y,s)dydxds \\
& \leq \frac{k_1^2}{2}\int_{0}^{t}\int_{0}^{\infty}\int_{0}^{\infty}(x^{\lambda}+y^{\lambda})(1+x)^{\mu}(1+y)^{\mu}f(x,s)f(y,s)dydxds \\
& \leq k_1^2\int_{0}^{t}\int_{0}^{\infty}\int_{0}^{\infty}x^{\lambda}(1+x)^{\mu}(1+y)^{\mu}f(x,s)f(y,s)dydxds \\
& \leq 2^{\mu} k_1^2 \max _{s \in [0,t]} \|f(s)\|\int_{0}^{t}[M_{\lambda}(f(s)) + M_{\rho}(f(s))]ds < \infty,
\end{align*}
  The above estimate implies that the second term on the left-hand side of (\ref{1st eq lambda}) is uniformly bounded which is independent of $n$ and therefore is convergent as $n \to \infty$. Hence, the LHS of equation (\ref{1st eq lambda}) converges as $n \to \infty$. This implies the convergence of the 
RHS in equation (\ref{1st eq lambda}).

%
Let us estimate the following integral for the last term on the RHS in (\ref{1st eq lambda}) as follows
\begin{align*}
 \int_{0}^{t}\int_{0}^{\infty}x^{\lambda}S(x)f(x,s)dxds & \leq m \int_{0}^{t}\int_{0}^{\infty}x^{\lambda}(1+x)^{1-\lambda}f(x,s)dxds\\
& \leq 2^{1-\lambda} m \int_{0}^{t} \bigg[\int_{0}^{1}x^{\lambda}f(x,s)dx+ \int_{1}^{\infty}xf(x,s)dx\bigg]ds\\
& \leq 2^{1-\lambda} m \int_{0}^{t} [M_{\lambda}(f(s))+M_{1}(f(s))]ds<\infty.
\end{align*}
Therefore, again as before, the last term on the RHS in (\ref{1st eq lambda}) is convergent as $n \to \infty$.
Since the remaining terms on the RHS in (\ref{1st eq lambda}) are non-negative, this implies that 
\begin{eqnarray}\label{term with b}
 \int_{0}^{t}\int_{0}^{\infty}\int_{0}^{x} y^{\lambda} b(x,y)S(x)f(x,s)dydxds < \infty.
\end{eqnarray}
Let us take the integral 
\begin{align*}
\hspace{-1.3cm} \int_{0}^{\infty}\int_{0}^{x} y^{\lambda} b(x,y)S(x)f(x,s)dydx=:\int_{0}^{\infty}R_xf(x,s)dx\ \ \mbox{where}\ \ R_x:= \int_{0}^{x}y^{\lambda} b(x,y)S(x)f(x,s)dy.
\end{align*}
Since $R_x$ is non-negative due to (A3). Then by using (A5), we have \
\begin{align}\label{Rx}
R_x &\geq L \int_{0}^{x} y^{\lambda} (1+x)^{\nu} dy= \frac{L}{\lambda+1}(1+x)^\nu x^{1+\lambda}\nonumber\\
&= \frac{L}{\lambda+1} \frac{(1+x)^{\nu+1}}{(\frac{1}{x}+1)}x^{\lambda}\geq \frac{L}{2(\lambda+1)}x^{\nu+ \lambda + 1} \hspace{.2cm} \mbox{for any} \hspace{.2cm} x\geq 1.
\end{align}
Substituting (\ref{Rx}) for $R_x$ and then into (\ref{term with b}), we obtain
\begin{align*}
 \frac{L}{2(\lambda+1)}\int_{0}^{t}\int_{0}^{\infty}x^{\nu+ \lambda + 1}f(x,s)dxds 
\leq \int_{0}^{t}\int_{0}^{\infty}\int_{0}^{x} y^{\lambda} b(x,y)S(x)f(x,s)dydxds < \infty.
\end{align*}
There are two cases. For $\rho - \mu <1$, we may take maximal $\lambda=\rho-\mu$ to give $I_{\rho + \nu - \mu+1}(t)< \infty$.
Otherwise, the condition $\lambda <1$ is more restrictive, i.e.\ we may take $\lambda=1-\delta$ for any $\delta>0$. 
This gives $I_{2+\nu-\delta}(t) < \infty$.
\end{proof}
\begin{proof} [Proof of Theorem~\ref{integrability thm}]
This can be proved by using the repeated application of Lemma \ref{lemma for integrability} as in \cite[Chapter 3, page 40]{Giri:thesis} and \cite{Costa:1995}.
%
\end{proof}
Now we prove the main result of the paper as follows.
\begin{proof}[Proof of Theorem~\ref{unique theorem}]
Let $f$ and $g$ be two solutions to (\ref{cfe1})-(\ref{in1}) on $[0,T[$ where $T>0$, with $f(0)=g(0)$, and set $Y=f-g$. For $n=1,2,3\ldots$ we define
\begin{eqnarray*}
u^n(t)=\int_{0}^{n}(1+x)|Y(x,t)|dx.
\end{eqnarray*}
Multiplying $|Y|$ by $(1+x)$, applying Fubini's Theorem to Definition \ref{def1} (iii) and the substitution $x'=x-y$, $y'=y$ in the first integral 
on the right-hand side, we obtain, for each $n$ and $0<t<T$,
\begin{align}\label{variable substitution}
u^n(t)=\int_{0}^{t}\int_{0}^{n}\int_{0}^{n-x}&\bigg[\frac{1}{2}(1+x+y)\text{sgn}(Y(x+y,s))-(1+x)\text{sgn}(Y(x,s))\bigg]\nonumber\\
       &\times K(x,y)\{f(x,s)f(y,s)-g(x,s)g(y,s)\}dydxds\nonumber\\
       -\int_{0}^{t}\int_{0}^{n}\int_{n-x}^{\infty}&(1+x)\text{sgn}(Y(x,s))K(x,y)\nonumber\\
&\times\{f(x,s)f(y,s)-g(x,s)g(y,s)\}dydxds\nonumber\\
       +\int_{0}^{t}\int_{0}^{n}\int_{x}^{\infty}&(1+x)\text{sgn}(Y(x,s))b(y,x)S(y)\{f(y,s)-g(y,s)\}dydxds\nonumber\\
       -\int_{0}^{t}\int_{0}^{n}&(1+x)\text{sgn}(Y(x,s))S(x)\{f(x,s)-g(x,s)\}dxds.
\end{align}
By interchanging the order of integration and interchanging the roles of $x$ and $y$, the symmetry of $K$ yields the identity
\begin{align}\label{intechanging}
\int_{0}^{n}&\int_{0}^{n-x}(1+x)\text{sgn}(Y(x,s))K(x,y)\{f(x,s)f(y,s)-g(x,s)g(y,s)\}dydx\nonumber\\
&=\int_{0}^{n}\int_{0}^{n-x}(1+y)\text{sgn}(Y(y,s))K(x,y)\{f(x,s)f(y,s)-g(x,s)g(y,s)\}dydx.
\end{align}
For $x,y>0$ and $t\in[0,T[$ we define the function $r$ by
\begin{align*}
r(x,y,t)=(1+x+y)\text{sgn}(Y(x+y,t))-(1+x)&\text{sgn}(Y(x,t))\\
&-(1+y)\text{sgn}(Y(y,t)).
\end{align*}
Using (\ref{intechanging}) we can show that (\ref{variable substitution}) can be rewritten as
\begin{align}\label{rewritten form}
u^n(t)=\frac{1}{2}&\int_{0}^{t}\int_{0}^{n}\int_{0}^{n-x}r(x,y,s)K(x,y)f(x,s)Y(y,s)dydxds\nonumber\\
+&\frac{1}{2}\int_{0}^{t}\int_{0}^{n}\int_{0}^{n-x}r(x,y,s)K(x,y)g(y,s)Y(x,s)dydxds\nonumber\\
+&\int_{0}^{t}\int_{0}^{n}\int_{x}^{\infty}x\text{sgn}(Y(x,s))b(y,x)S(y)Y(y,s)dydxds\nonumber\\
-&\int_{0}^{t}\int_{0}^{n}x\text{sgn}(Y(x,s))S(x)Y(x,s)dxds\nonumber\\
+&\int_{0}^{t}\int_{0}^{n}\int_{x}^{\infty}\text{sgn}(Y(x,s))b(y,x)S(y)Y(y,s)dydxds\nonumber\\
-&\int_{0}^{t}\int_{0}^{n}\text{sgn}(Y(x,s))S(x)Y(x,s)dxds\nonumber\\
-&\int_{0}^{t}\int_{0}^{n}\int_{n-x}^{\infty}(1+x)\text{sgn}(Y(x,s))K(x,y)\nonumber\\
&\times\{f(x,s)Y(y,s)+g(y,s)Y(x,s)\}dydxds.
\end{align}
Since the sixth integral and the last term in the seventh integral on the right-hand side of (\ref{rewritten form}) are non-negative, we may omit them. Thus we obtain, by interchanging the order of integration for the fifth integral,
\begin{align}\label{main equation}
u^n(t)\leq &\frac{1}{2}\int_{0}^{t}\int_{0}^{n}\int_{0}^{n-x}r(x,y,s)K(x,y)f(x,s)Y(y,s)dydxds\nonumber\\
&+\frac{1}{2}\int_{0}^{t}\int_{0}^{n}\int_{0}^{n-x}r(x,y,s)K(x,y)g(y,s)Y(x,s)dydxds\nonumber\\
&+\int_{0}^{t}\int_{0}^{n}\int_{x}^{\infty}xb(y,x)S(y)|Y(y,s)|dydxds\nonumber\\
&-\int_{0}^{t}\int_{0}^{n}xS(x)|Y(x,s)|dxds\nonumber\\
&+\int_{0}^{t}\int_{0}^{n}\int_{0}^{y}b(y,x)S(y)|Y(y,s)|dxdyds\nonumber\\
&+\int_{0}^{t}\int_{n}^{\infty}\int_{0}^{n}b(y,x)S(y)|Y(y,s)|dxdyds\nonumber\\
&-\int_{0}^{t}\int_{0}^{n}\int_{n-x}^{\infty}(1+x)\text{sgn}(Y(x,s))K(x,y)f(x,s)Y(y,s)dydxds\nonumber\\
=:&\int_{0}^{t}\sum_{i=1}^{7}S_{i}^n(s)ds.
\end{align}
Here $S_i^n$, for $i=1,\ldots 7,$ are the corresponding integrands in the preceding lines.\\
We now consider each $S_i^n$ individually. Noting that for all $q,q_1, q_2 \in \mathbb{R}$, we have $\text{sgn}(q_1) \text{sgn}(q_2)= \text{sgn}(q_{1} q_{2})$ and $|q|=q \text{sgn}(q)$.
We find that
\begin{align}\label{ineq}
r(x,y,s)Y(y,s)\leq [(1+x+y)+(1+x)-(1+y)]|Y(y,s)|\leq 2(1+x)|Y(y,s)|.
\end{align}
By using (A2) and (\ref{ineq}), let us consider the integral with $S_1^n$ in (\ref{main equation}).
\begin{align*}
\int_{0}^{t}S_1^n(s)ds
                      &\leq k_1^2\int_{0}^{t}\int_{0}^{n}\int_{0}^{n-x}(1+x)^{1+\mu}(1+y)^{\mu}f(x,s)|Y(y,s)|dydxds\nonumber\\
                      &\leq k_1^2\int_{0}^{t}\bigg[\int_{0}^{1}(1+x)^{1+\mu}f(x,s)dx+\int_{1}^{n}x^{1+\mu}(\frac{1}{x}+1)^{1+\mu}f(x,s)dx\bigg]u^n(s)ds\nonumber\\
                      &=2^{1+\mu}k_1^2\int_{0}^{t}[M_{0}(f(s))+ M_{1+\mu}(f(s))]u^n(s)ds\nonumber\\
                      &=\int_{0}^{t}\varphi_f(s)u^n(s)ds
\end{align*}
where $\varphi_f(s):=2^{1+\mu}k_1^2[M_{0}(f(s))+ M_{1+\mu}(f(s))]$.
Similarly, by defining $\varphi_g(s):=2^{1+\mu}k_1^2 [M_{0}(g(s))+M_{1+\mu}(g(s))]$, we estimate
\begin{eqnarray*}
\int_{0}^{t}S_{2}^n(s)ds\leq \int_{0}^{t}\varphi_g(s)u^n(s)ds.
\end{eqnarray*}
Next, by using (\ref{mass1}) in the integral with $S_4^n$, we calculate third and fourth integrals in (\ref{main equation}) together. Then 
we interchange the roles of $x$ and $y$, and change the order of integrations, respectively, in the term with $S_4^n$ to obtain 
\begin{align}\label{S3S42}
\int_{0}^{t} (S_3^n(s)+ S_4^n(s))ds=&\int_{0}^{t}\int_{0}^{n}\int_{x}^{\infty}xb(y,x)S(y)|Y(y,s)|dydxds\nonumber\\
&-\int_{0}^{t}\int_{0}^{n}\int_{x}^{n}xb(y,x)S(y)|Y(y,s)|dydxds.
\end{align}

Let us evaluate the following integral by using (\ref{mass1}) and (A4).
\begin{align*}
\int_{0}^{t}\int_{0}^{\infty}\int_{x}^{\infty}xb(y,x)S(y)&|Y(y,s)|dydx\\ =&\int_{0}^{t}\int_{0}^{\infty}\int_{0}^{y}xb(y,x)S(y)|Y(y,s)|dxdy ds\\
\leq &m \int_{0}^{t}\bigg[\int_{0}^{1}y(1+y)^{1-\lambda}\{f(y,s)+g(y,s)\}dy \\ 
&\hspace{.5cm} +\int_{1}^{\infty} y^{2-\lambda}(1+1/y)^{1-\lambda}\{f(y,s)+g(y,s)\}dy\bigg]ds\\
\leq &2^{1-\lambda} m\int_{0}^{t}[\{M_1(f(s)) + M_1(g(s))\}\\
&\hspace{.6cm}+ \{M_{2-\lambda}(f(s)) + M_{2-\lambda}(g(s))\}] ds< \infty.
\end{align*}
Note that we have used the integrability of higher moments of $f$ and $g$ from Theorem \ref{integrability thm}. Therefore, from (\ref{S3S42}) and the finiteness of the above integral, we have 
\begin{eqnarray}\label{S3S4_conv}
\int_{0}^{t} (S_3^n(s)+ S_4^n(s))ds \to 0 \hspace{.5cm} \mbox{as} \hspace{.5cm} n \to \infty.
\end{eqnarray}
Now let us consider the integral with $S_5^n$ in (\ref{main equation}).
By interchanging the roles of $x$ and $y$, and using (\ref{N1}) and (A4), we obtain 
\begin{align*}
 \hspace{-.4cm}\int_{0}^{t}S_5^n(s)ds&=\int_{0}^{t}\int_{0}^{n}\int_{0}^{x}b(x,y)S(x)|Y(x,s)|dydxds\leq mN\int_{0}^{t}\int_{0}^{n} (1+x)^{1-\lambda} |Y(x,s)|dxds\nonumber\\
& \leq L \int_{0}^{t}u^n(s)ds, \hspace{.2cm} \mbox{where} \hspace{.2cm} L=mN.
\end{align*}
Thus, we estimate
\begin{eqnarray}\label{main1}
\int_{0}^{t}\bigg[S_{1}^n(s)+S_{2}^n(s)+S_{5}^n(s)\bigg]ds \leq \int_{0}^{t}\varphi(s)u^n(s)ds
\end{eqnarray}
where $\varphi(s)=\varphi_f(s)+\varphi_g(s)+L$ is integrable by Theorem \ref{integrability thm}.

Next, to solve the integral with $S_6^n$ in (\ref{main equation}), we interchange the roles of $x$ and $y$, 
and then use (\ref{N1}) and (A4) to estimate the following integral for each $s \in [0,t]$
\begin{align}\label{S6}
\int_{n}^{\infty}\int_{0}^{n}b(x,y)S(x)|Y(x,s)|dydx 
 \leq mN \int_{n}^{\infty} (1+x)[f(x,s)+g(x,s)]dydx.
\end{align}
The right-hand side of (\ref{S6}) is always bounded by the constant $mN \sup_{s\in[0,t]}[\|f(s)\|+\|g(s)\|]$ 
and therefore the dominated convergence theorem leads to 
\begin{eqnarray}\label{S6_lim}
 \int_{0}^{t}S_6^n(s)ds \to 0 \hspace{.2cm} as \hspace{.2cm} n\to \infty.
\end{eqnarray}
To consider the integral with $S_7^n$ in (\ref{main equation}), by using (A2), we first observe that
\begin{align*}
\bigg|\int_{0}^{t}S_7^n(s)ds\bigg| 
&\leq k_1^2\int_{0}^{t}\int_{0}^{\infty}\int_{0}^{\infty}(1+x)^{1+\mu}(1+y)^{\mu}f(x,s)|Y(y,s)|dydx\\
&\leq k_1^2\sup_{s \in [0,t]}[\|f(s)\|+\|g(s)\|]\\
&\hspace{.2cm}\times\int_{0}^{t}\bigg[\int_{0}^{1}(1+x)^{1+\mu}f(x,s)dx+\int_{1}^{\infty}(\frac{1}{x}+1)^{1+\mu}x^{1+\mu}f(x,s)dx\bigg]ds\\
& \leq 2^{1+\mu}k_1^2\sup_{s \in [0,t]}[\|f(s)\|+\|g(s)\|]\int_{0}^{t}(M_{0}(f(s))+M_{1+\mu}(f(s)))ds< \infty.
\end{align*}
Thus, by Lemma 1.2 in \cite{Giri:2011II} we obtain
\begin{eqnarray}\label{S7}
\int_{0}^{t}S_7^n(s)ds\to 0 \ \ \text{as}\ \ n\to \infty.
\end{eqnarray}
The sequence $u^n$ is bounded and monotone. Thus, from (\ref{S3S4_conv}), (\ref{main1}), (\ref{S6_lim}) and (\ref{S7}) we obtain
\begin{align*}
u(t):&=\int_{0}^{\infty}(1+x)|Y(x,t)|dx=\lim_{n\to \infty}u^n(t)\\
&\leq \lim_{n\to \infty}\int_{0}^{t}\varphi(s)u^n(s)ds +\lim_{n\to \infty}\int_{0}^{t}[S_{3}^n(s)+S_{4}^n(s)]ds
+\lim_{n\to \infty}\int_{0}^{t}[S_{6}^n(s)+S_{7}^n(s)]ds\\
&=\int_{0}^{t}\varphi(s)\int_{0}^{\infty}(1+x)|Y(x,s)|dxds,
\end{align*}
which can be rewritten as
\begin{eqnarray*}
u(t)\leq \int_{0}^{t}\varphi(s)u(s)ds
\end{eqnarray*}
with $\varphi(s)\geq 0$. Then an application of Gronwall's inequality gives
\begin{eqnarray*}
u(t)=\int_{0}^{\infty}(1+x)|Y(x,t)|dx = 0 \ \ \mbox{for all}\ \ t\in[0,T).
\end{eqnarray*}
Therefore, we have
\begin{eqnarray*}
f(x,t)=g(x,t)\ \ \mbox{for a.e.}\ \ x\in\mathbb{R}_{\geq 0}.
\end{eqnarray*}
\end{proof}

 \section*{Acknowledgements}
 The author wants to thank a reviewer for his comments and suggestions that helped to improve the results in the manuscript.

 \end{document}